\documentclass[reqno]{amsart}
\usepackage{amsmath,amsthm,amscd,amssymb,amsfonts, amsbsy}
\usepackage{latexsym, color}
\usepackage{pxfonts}
\usepackage{mathrsfs}
\usepackage{bbm}
\usepackage{mathtools}
\usepackage{enumitem}
\usepackage{todonotes}
\usepackage{marginnote}

\theoremstyle{plain}
\newtheorem{theorem}[equation]{Theorem}
\newtheorem{lemma}[equation]{Lemma}

\newtheorem{proposition}[equation]{Proposition}

\theoremstyle{definition}
\newtheorem{definition}[equation]{Definition}
\newtheorem{condition}[equation]{Condition}

\theoremstyle{remark}
\newtheorem{remark}[equation]{Remark}

\newcommand{\capacity}{\operatorname{cap}}

\newcommand{\dist}{\operatorname{dist}}
\newcommand{\diam}{\operatorname{diam}}
\newcommand{\intr}{\operatorname{int}}

\numberwithin{equation}{section}

\newcommand{\bC}{\mathbb{C}}

\newcommand{\bR}{\mathbb{R}}

\providecommand{\set}[1]{\{#1\}}
\providecommand{\Set}[1]{\left\{#1\right\}}

\providecommand{\abs}[1]{\lvert#1\rvert}

\providecommand{\norm}[1]{\lVert#1\rVert}

\renewcommand{\vec}[1]{\boldsymbol{#1}}

\DeclareMathOperator*{\osc}{osc}

\begin{document}
\title[Dirichlet problem for elliptic equations in non-divergence form in $\bR^2$]
{The Dirichlet problem for second-order elliptic equations in non-divergence form with continuous coefficients: The two-dimensional case}

\author[H. Dong]{Hongjie Dong}
\address[H. Dong]{Division of Applied Mathematics, Brown University,
182 George Street, Providence, RI 02912, United States of America}
\email{Hongjie\_Dong@brown.edu}
\thanks{H. Dong was partially supported by the NSF under agreement DMS-2350129.}

\author[D. Kim]{Dong-ha Kim}
\address[D. Kim]{Department of Mathematics, Chung-Ang University, 84 Heukseok-ro, Dongjak-gu, Seoul 06974, Republic of Korea}
\email{kimdongha91@cau.ac.kr}

\author[S. Kim]{Seick Kim}
\address[S. Kim]{Department of Mathematics, Yonsei University, 50 Yonsei-ro, Seodaemun-gu, Seoul 03722, Republic of Korea}
\email{kimseick@yonsei.ac.kr}
\thanks{S. Kim is supported by the National Research Foundation of Korea (NRF) under agreement NRF-2022R1A2C1003322.}



\begin{abstract}
This paper investigates the Dirichlet problem for a non-divergence form elliptic operator $L$ in a bounded domain of $\mathbb{R}^2$.
Assuming that the principal coefficients satisfy the Dini mean oscillation condition, we establish the equivalence between regular points for $L$ and those for the Laplace operator.
This result closes a gap left in the authors' recent work on higher-dimensional cases (Math. Ann. 392(1): 573--618, 2025).
Furthermore, we construct the Green's function for $L$ in regular two-dimensional domains, extending a result by Dong and Kim (SIAM J. Math. Anal. 53(4): 4637--4656, 2021).
\end{abstract}
%\today
\maketitle

%------------------------------------------------------------------------------%
\section{Introduction}
%------------------------------------------------------------------------------%

This article extends our recent work \cite{DKK25} to the two-dimensional setting.
In $\bR^2$, we consider the elliptic operator  $L$ given by
\[
Lu=  \sum_{i,j=1}^2 a^{ij} D_{ij}u+ \sum_{i=1}^2 b^i D_i u+cu.
\]
We assume that the coefficient matrix $\mathbf{A}=(a^{ij})$ is symmetric, satisfies the uniform ellipticity condition, and has Dini mean oscillation.
Additionally, we assume that the lower-order coefficients satisfy $\vec b=(b^1,b^2) \in L^{p_0}_{\rm loc}(\bR^2)$,  $c \in L^{p_0/2}_{\rm loc}(\bR^2)$, for some $p_0>2$, with $c \le 0$.

This paper focuses on the Dirichlet problem for the equation $Lu=0$ in $\Omega \subset \mathbb{R}^2$ with continuous boundary data $u=\varphi$ on $\partial \Omega$.
The two-dimensional case was not considered in \cite{DKK25}, primarily due to the fundamentally different nature of the Green's function in two dimensions. Compared to the case  $d \geq 3$, the Green's function in $\bR^2$ exhibits a logarithmic singularity at the pole, requiring a different approach (see \cite{DK09, DK21}).
Another key difference arises when comparing Green's functions in two and higher dimensions.
Consider the Green's function for the Laplace operator.
In $\bR^d$ for $d\ge 3$, the Green's function for the whole space provides a uniform upper bound for the Green's function in any bounded domain $\Omega \subset \bR^d$:
\[
G_\Omega(x,y) \le G_{\bR^d}(x,y) =\frac{1}{d(d-2)\omega_d} \,\frac{1}{\abs{x-y}^{d-2}}.
\]
This function, $G_{\bR^d}(x,y)$, is often referred to as the fundamental solution.
In contrast, no Green's function exists for the entire space $\bR^2$.
The well-known function,
 \[
 -\frac{1}{2\pi} \log \,\abs{x-y},
 \]
 is not actually the Green's function for the Laplace operator in $\bR^2$, as it changes sign and fails to provide a uniform upper bound for the Green's function in a domain $\Omega \subset \bR^2$.
This added complexity makes the analysis in the two-dimensional setting significantly more delicate.
On the other hand, since $\bR^2$ can be identified with $\bC$, many techniques from complex function theory, which are unavailable in higher dimensions, can be employed to study the Dirichlet problem for Laplace equations in $\bR^2$. See, for instance, the excellent book by Garnett and Marshall \cite{GM2005}.

A foundational result in two-dimensional potential theory for harmonic functions is rooted in the Riemann mapping theorem: if $\Omega \subset \mathbb{R}^2$ is the interior of a closed Jordan curve $\Gamma$, then every boundary point of $\Omega$ is regular. Riemann initially believed that this result held for all simply connected domains. However, his proof relied on the Dirichlet principle, which requires specific assumptions about the boundary (e.g., that it is a Jordan curve).
In 1900, Osgood provided a more precise topological criterion for regularity. He showed that a boundary point $z_0 \in \partial\Omega$ is regular if it belongs to a connected component of $\mathbb{R}^2 \setminus \Omega$ that contains at least one other point distinct from $z_0$.
His proof relies on the construction of a barrier function involving a branch of $\log (z-z_0)$.

A major breakthrough came in 1924 when Wiener established a celebrated necessary and sufficient condition for the regularity of a boundary point with respect to the Laplace operator. This condition, now known as Wiener's test, applies to general domains in $\bR^n$, including the two-dimensional case.
The study of regularity was further extended to more general elliptic equations. For divergence form elliptic equations, Littman, Stampacchia, and Weinberger \cite{LSW63} made a fundamental contribution by demonstrating the equivalence of regular boundary points for equations with bounded and measurable coefficients.
A comprehensive treatment of divergence form elliptic operators with possibly non-symmetric coefficients was later given by Gr\"uter and Widman \cite{GW82}.

However, the situation for non-divergence form equations is more delicate. For non-divergence form elliptic equations, the equivalence of regular boundary points with those of the Laplace equation has been established under progressively weaker assumptions on the coefficients.
Oleinik \cite{Oleinik} first proved this equivalence for $C^{3,\alpha}$ coefficients, while Herv\'e \cite{Herve} extended the result to H\"older continuous coefficients. Later, Krylov \cite{Krylov67} generalized the result to the case of Dini continuous coefficients.
Miller \cite{Miller} provided a counterexample demonstrating that the equivalence of regular boundary points can fail even when the coefficients are uniformly continuous (see also \cite{Bauman84a, KY17}).
Bauman \cite{Bauman85} extended Wiener's result to elliptic operators in non-divergence form with continuous coefficients, providing a Wiener-type criterion for the regularity of boundary points. However, this work did not offer new insights into the alignment of regular points for the operator and the Laplacian.
This highlights that, for non-divergence form equations, some form of modulus of continuity assumption on the coefficients is necessary to ensure equivalence.

We show that under the assumption that the coefficients of $L$ satisfy Conditions \ref{cond1} and \ref{cond2}, the regular points for $L$ coincide with those for the Laplace operator, as established in Theorem \ref{thm0800sat}.
This result enables us to prove the unique solvability of the Dirichlet problem with continuous boundary data in regular domains (Theorem \ref{thm0802sat}).
Furthermore, we construct the Green's function for $L$ in regular two-dimensional domains and establish pointwise bounds for it (Theorem \ref{thm1127sat}).
This represents a significant advancement, as the existence and estimates of Green's function were previously known only for operators without lower-order terms in $C^{2,\alpha}$ domains.
In $C^{1,1}$ domains, earlier results required the principal coefficients $\mathbf A$ to satisfy an $L^2$-Dini mean oscillation condition (see \cite{DK21}), which is more restrictive than the Dini mean oscillation condition used here and previously studied in \cite{DK17, DEK18}.
We also note that in higher dimensions, such a restrictive condition was not required (see \cite{HK20}), underscoring the additional technical challenges in the two-dimensional case.

The article is organized as follows. In Section \ref{sec2}, we introduce key concepts and construct Green's functions for two-dimensional balls, along with useful estimates. In Section \ref{sec3}, we define relative potential and capacity, establishing some of their fundamental properties. This section corresponds to Section 5 of \cite{DKK25}, but with notable adaptations for the two-dimensional setting. Finally, in Section \ref{sec4}, we establish the Wiener criterion (Theorem \ref{thm_wiener}) and provide the proofs of our main results.

\section{Main results}
The following theorems summarize our main results. While they are repeated from the main text, we present them here for improved readability and convenience.

\begin{theorem}[Theorem \ref{thm0800sat}]
Assume Conditions \ref{cond1} and \ref{cond2}, and let $\Omega$ be a bounded open domain in $\mathbb{R}^2$.
A point $x_0 \in \partial \Omega$ is a regular point for $L$ if and only if $x_0$ is a regular point for the Laplace operator.
\end{theorem}

\begin{theorem}[Theorem \ref{thm0802sat}]
Assume that Conditions \ref{cond1} and \ref{cond2} hold.
Let $\Omega$ be a bounded regular domain in $\mathbb{R}^2$.
For $f \in C(\partial\Omega)$, the Dirichlet problem,
\[
Lu=0 \;\text{ in }\;\Omega, \quad u=f \;\text{ on }\;\partial\Omega,
\]
has a unique solution $u \in W^{2, p_0/2}_{\rm loc}(\Omega) \cap C(\overline\Omega)$.
\end{theorem}

\begin{theorem}[Theorem	\ref{thm1127sat}]
Under Conditions \ref{cond1} and \ref{cond2}, let $\Omega \subset \bR^2$ be a bounded regular domain contained in $\mathcal{B}$.
Then, there exists a Green's function $G(x,y)$ in $\Omega$, which satisfies the following pointwise bound:
\[
0\le G(x,y) \le C \left\{1+ \log \left(\frac{\diam \Omega}{\abs{x-y}}\right)\right\},\quad x\neq y \in \Omega.
\]
where $C$ is a constant depending only on $\lambda$, $\Lambda$, $p_0$, $\diam \mathcal{B}$.
\end{theorem}

\begin{theorem}[Theorem	\ref{thm0738}]
Assume that Conditions \ref{cond1} and \ref{cond2} hold.
Let $\Omega$ be a bounded regular domain in $\mathbb{R}^2$, and let $f \in C(\partial \Omega)$.
Consider the Dirichlet problem
\[
Lu=0 \;\text{ in }\;\Omega, \quad u=f \;\text{ on }\;\partial\Omega,
\]
where $u \in W^{2, p_0/2}_{\rm loc}(\Omega) \cap C(\overline\Omega)$ is the solution.
\begin{enumerate}[leftmargin=*]
\item[(i)]
Suppose $x_0 \in \partial \Omega$ is in a connected component of $\mathbb{R}^2 \setminus \Omega$ that contains at least one other point distinct from $x_0$.
If $f$ is H\"older continuous at $x_0$, then $u$ is also H\"older continuous at $x_0$, possibly with a different H\"older exponent.
\item[(ii)]
Suppose there exists a constant $r_0>0$ such that every point $x_0 \in \partial \Omega$ is in a connected component of $\mathbb{R}^2 \setminus \Omega$ that contains a point at least  $r_0$ away from $x_0$.
If $f \in C^{\beta}(\partial\Omega)$ for some $\beta \in (0,1)$, then $u \in C^\alpha(\overline\Omega)$ for some $\alpha \in (0,\beta)$.
\end{enumerate}
\end{theorem}

%------------------------------------------------------------------------------%
\section{Preliminary}			\label{sec2}
%------------------------------------------------------------------------------%

We choose an open ball $\mathcal{B}=B_{2R}(0)$ such that $B_R(0) \Supset \overline \Omega$, where $\Omega$ is a domain under consideration.
Throughout the paper, this ball $\mathcal{B}$ remains fixed. 
Also, we adopt the standard summation convention over repeated indices.

\begin{condition}			\label{cond1}
The coefficients of $L$ are measurable and defined in the whole space $\bR^2$.
The principal coefficients matrix $\mathbf{A}=(a^{ij})$ is symmetric and satisfies the ellipticity condition:
\[
\lambda \abs{\xi}^2 \le \mathbf{A}(x) \xi  \cdot \xi \le \lambda^{-1} \abs{\xi}^2,\quad \forall x \in \bR^2,\;\;\forall \xi \in \bR^2,
\]
where $\lambda \in (0,1]$ is a constant.
The lower-order coefficients $\vec b=(b^1, b^2)$ and $c$ belong to $L^{p_0/2}_{\rm loc}(\bR^2)$ and $L^{p_0}_{\rm loc}(\bR^2)$ for some $p_0>2$, and
\[
\norm{\vec b}_{L^{p_0}(\mathcal{B})}+ \norm{c}_{L^{p_0/2}(\mathcal{B})} \le \Lambda,
\]
where $\Lambda=\Lambda(\mathcal{B})<\infty$.
Additionally, we assume that $c\le 0$.
\end{condition}

\begin{condition}			\label{cond2}
The mean oscillation function $\omega_{\mathbf A}: \bR_+ \to \bR$ defined by
\[
\omega_{\mathbf A}(r):=\sup_{x\in \mathcal{B}} \fint_{\mathcal{B} \cap B_r(x)} \,\abs{\mathbf A(y)-\bar {\mathbf A}_{x,r}}\,dy, \;\; \text{where }\bar{\mathbf A}_{x,r} :=\fint_{\mathcal{B} \cap B_r(x)} \mathbf A,
\]
satisfies the Dini condition, i.e.,
\[
\int_0^1 \frac{\omega_{\mathbf A}(t)}t \,dt <+\infty.
\]
\end{condition}

We begin by constructing the Green's function for the operator $L$ and establishing its pointwise estimate in two-dimensional balls. Notably, the following theorem applies to Green's function for all balls contained in $\mathcal{B}$, including $\mathcal{B}$ itself.  

\begin{theorem}	\label{thm2.3}
Assume that Conditions \ref{cond1} and \ref{cond2} hold.
Let $B_r=B_r(x_0) \subset \mathcal{B}$.
Then, there exists a Green's function $G(x,y)$ of $L$ in $B_r$ and the Green's function is unique in the following sense: if $v$ is the unique adjoint solution of the problem
\begin{equation}				\label{eq1738sun}
L^*v=D_{ij}(a^{ij}v)-D_i(b^iv)+cv=f\;\text{ in }\;B_r,\quad v=0\;\text{ on }\;\partial B_r,
\end{equation}
where $f \in L^p(B_r)$ with $p>1$, then $v$ is represented by
\[
v(y)=-\int_{B_r} G(x,y) f(x)\,dx.
\]
Green's function $G(x,y)$ satisfies the pointwise estimate:
\begin{equation}			\label{1553thu}
0 \le G(x,y) \le C\left(1+ \log \frac{2r}{\abs{x-y}}\right),\quad  x \neq y \in B_r,
\end{equation}
where $C$ depends only on $\lambda$, $\Lambda$, $p_0$, $\omega_{\mathbf A}$, and $\diam \mathcal{B}$.
Moreover, the function
\[
G^*(x,y):=G(y,x)
\]
is the Green's function for the adjoint operator $L^*$.
\end{theorem}

\begin{proof}
We adapt the proof of \cite[Theorem 4.1]{DKK25}.
In \cite{DK21}, Green's functions for the operator $L_0:= a^{ij}D_{ij}$  are constructed in 
$C^{2,\alpha}$ domains, which in particular include $B_r$.
To construct the Green's function for the operator $L=a^{ij}D_{ij}+b^i D_i+c$, we consider the following problem for each $y \in B_r$:
\begin{equation}			\label{eq1123thu}
Lu= -b^iD_i G_0(\cdot,y) -c G_0(\cdot,y)\;\text{ in }\;B_{r},\quad u=0 \;\text{ on }\;\partial B_{r},
\end{equation}
where $G_0(x, y)$ is the Green's function for $L_0$ in $B_r$.
We invert the sign of the Green's function $G_0$ so that
\[
L_0 G_0(\,\cdot,y)=-\delta_y\;\text{ in }\;B_r \quad\text{and}\quad G_0(\,\cdot,y)=0 \;\text{ on }\;\partial B_r.
\]
This ensures that $G_0$ is nonnegative.

Note that the Green's function estimates from \cite{DK21} yield
\begin{equation}	\label{bound1939mon}
G_0(x,y) \le C\left(1+ \log \frac{2r}{\abs{x-y}}\right) \le C\left(1+ \log \frac{\diam \mathcal B}{\abs{x-y}}\right),\quad
\abs{D_xG_0(x,y)} \le \frac{C}{\abs{x-y}},
\end{equation}
where $C$ depends only on $\lambda$, $\omega_{\mathbf A}$, and $\diam \mathcal{B}$.
While \cite{DK21} states that $C$ also depends on the domain (which, in our case, is $B_r$), a scaling argument shows that this dependence arises solely through $\diam \mathcal{B}$, which provides an upper bound for $r$.

Note that $G(\,\cdot,y) \in L^p(B_r)$ for all $p<\infty$ and $D G(\,\cdot, y) \in L^p(B_r)$ for all $p<2$. 
Thus, there exists $p_1>1$, determined by $p_0$, such that for each $y \in B_r$, we have
\begin{equation}			\label{eq1818sat}
\norm{b^i D_iG_0(\,\cdot,y)}_{L^{p_1}(B_r)} + \norm{cG_0(\,\cdot, y)}_{L^{p_1}(B_r)} \le C,
\end{equation}
where $C$ depends only on $\lambda$, $\Lambda$, $p_0$, $\omega_{\mathbf A}$, and $\diam \mathcal{B}$.
Therefore, by \cite[Theorem 4.2]{Krylov21}, there exists a unique solution
\[
u=u^y \in W^{2,p_1}(B_r)\cap W^{1,p_1}_0(B_r)
\]
to the problem \eqref{eq1123thu}.
Applying the Sobolev embedding theorem and $L^p$ estimates, we deduce from \eqref{eq1818sat} that 
\begin{equation}	\label{eq1410thu}
\norm{u^y}_{L^\infty(B_r)} \le \norm{u^y}_{W^{2,p_1}(B_r)}\le C,
\end{equation}
where $C$ depends only on $\lambda$, $\Lambda$, $p_0$, $\omega_{\mathbf A}$, and $\diam \mathcal{B}$.
In particular, $C$ is uniform for all $r$ and $y$.
Now, we will demonstrate that
\[
G(x,y):=G_0(x,y)+u^y(x)
\]
serves as the Green's function for $L$ in $B_r$.
For any $f \in C^\infty_c(B_r)$, let $v \in L^{p_1'}(B_r)$ be the solution of \eqref{eq1738sun}.
By \cite[Theorem 1.8]{DEK18}, we find that $v \in C(\overline B_r)$.
Moreover, from the definition of the solution to the problem \eqref{eq1738sun}, we have
\begin{equation}			\label{eq1428wed}
\int_{B_r} f w = \int_{B_r} v Lw,\quad \forall w \in  W^{2,p_1}(B_r)\cap W^{1,p_1}_0(B_r).
\end{equation}
Since $u^y \in W^{2,p_1}(B_r)\cap W^{1,p_1}_0(B_r)$ is a solution of \eqref{eq1123thu}, it follows that
\begin{equation}			\label{eq1805sun}
\int_{B_r} f u^y = \int_{B_r} v Lu^y = -\int_{B_r} b^i D_i G_0(\,\cdot, y)v - \int_{B_r} c G_0(\,\cdot, y)v.
\end{equation}
On the other hand, taking  $w=G_0^\epsilon(\,\cdot, y)$ in \eqref{eq1428wed}, where $G_0^\epsilon(\,\cdot, y)$ is the approximate Green's function of $L_0$ in $B_r$, i.e.,
\[
L_0 G_0^\epsilon(\,\cdot,y)=-\frac{1}{\abs{B_{\epsilon}(y)}} \mathbbm{1}_{B_{\epsilon}(y)}\;\text{ in }\;B_{r},\quad G^{\epsilon}_0(\,\cdot,y)=0 \;\text{ on }\;\partial B_{r},
\]
as considered in \cite{DK21} (with the sign inverted), we obtain 
\begin{multline}			\label{eq2103sat}
\int_{B_r} f G_0^\epsilon(\,\cdot,y)=\int_{B_r} v L G_0^\epsilon(\,\cdot,y)\\
=-\fint_{B_\epsilon(y)\cap B_r} v + \int_{B_r} b^i D_i G_0^\epsilon(\,\cdot, y)v+\int_{B_r} c G_0^\epsilon(\,\cdot, y)v.
\end{multline}

Next, we present a lemma analogous to \cite[Lemma 4.8]{DKK25}.
\begin{lemma}		\label{lem1631wed}
There exists a sequence $\{\epsilon_k\}$ converging to zero such that
\begin{align*}
G^{\epsilon_k}_0(\,\cdot, y) &\rightharpoonup G_0(\,\cdot, y) \;\text{ weakly in } L^p(B_r)\; \text{ for }\; 1<p<\infty,\\
D G^{\epsilon_k}_0(\,\cdot, y) &\rightharpoonup DG_0(\,\cdot, y) \;\text { weakly in } L^p(B_r)\; \text{ for }\; 1<p<2.
\end{align*}
\end{lemma}

\begin{proof}
The first part follows from the following fact from \cite{DK21}:
\[
G^{\epsilon_k}_0(\,\cdot, y) \rightharpoonup G_0(\,\cdot, y) \;\text{in the weak-$*$ topology of }\; \mathrm{BMO}(B_r).
\]
For the proof for the second part, refer to \cite[Lemma 4.8]{DKK25}.
\end{proof}

Therefore, by taking the limit $\epsilon_k \to 0$ in \eqref{eq2103sat}, we obtain
\begin{equation}			\label{eq1806sun}
\int_{B_r} f G_0(\,\cdot,y) =-v(y) + \int_{B_r} b^i D_i G_0(\,\cdot, y)v+\int_{B_r} c G_0(\,\cdot, y)v,
\end{equation}
where we utilized Lemma \ref{lem1631wed} and the fact that  $v \in C(\overline{\mathcal{B}})$.
By combining \eqref{eq1805sun} and \eqref{eq1806sun}, we deduce
\[
\int_{B_r} f G(\,\cdot, y)= \int_{B_r} f G_0(\,\cdot, y) + \int_{B_r} f u^y= -v(y).
\]
Therefore, we conclude that $G(x, y)$ is the Green's function for $L$ in $B_r$.

Combining \eqref{bound1939mon} with the uniform bound of $u^y$ from \eqref{eq1410thu}, we obtain the pointwise estimate \eqref{1553thu}.
Moreover,
\[
G^*(x,y):=G(y,x)
\]
is the Green's function for $L^*$ in $B_r$.
See \cite[Theorem 4.1]{DKK25} for the details.
\end{proof}

We also obtain the lower bound for the Green's function when $x$ and $y$ are sufficiently far away from the boundary.

\begin{theorem}	\label{thm_green_function}
Assume that Conditions \ref{cond1} and \ref{cond2} hold.
Let $B_{4r}=B_{4r}(x_0) \subset \mathcal{B}$, and let $G(x, y)$ be the Green's function for $L$ in $B_{4r}$.
There exists a constant $C_0>1$, depending only on $\lambda$, $\Lambda$, $\omega_{\mathbf A}$, and $\diam \mathcal B$, such that the following estimate holds:
\begin{equation}		\label{estimate_green}
\frac{1}{C_0} \log \left(\frac{3r}{\abs{x-y}}\right) \le G(x,y)\le C_0 \log \left(\frac{3r}{\abs{x-y}}\right),\quad x \neq y \in \overline B_r.
\end{equation}
\end{theorem}

\begin{proof}
The upper bound is already provided in \eqref{1553thu} as $\abs{x-y} \le 2r$.
To establish the lower bound, we fix $x \neq y \in B_r$, and set $\rho=\abs{x-y}$.
We consider a collection of balls $B_{3^{j}\rho}(y)$ for $i = 1, \ldots, N$, such that $B_{3^{N}\rho}(y) \subset B_{4r}$ but $B_{3^{N+1}\rho}(y)\not \subset B_{4r}$.
Note that the chosen integer $N$ satisfies
\begin{equation}\label{mon1543}
N+1 \simeq \log \left(\frac{3r}{\abs{x-y}}\right).
\end{equation}

We use the notation $G_{3^i \rho}$ for the Green's function for $L$ on $B_{3^i \rho}(y)$.
Our first claim is that there exist a positive constant $C_1$ such that we have
\begin{equation}		\label{lower_1229}
C_1 \le G_{3^i \rho}(z,y),\quad \forall z \in \partial B_{3^{i-1}\rho}(y),\quad i=1,2,\ldots, N.
\end{equation}
Here, the constant $C_1$ only depends on $\lambda$, $\Lambda$, $\omega_{\mathbf{A}}$, and $\diam \mathcal{B}$.

Assume for now that the claim is proven.
We rewrite $G(x,y)$ as
\begin{equation}		\label{mon_240818}
G(x,y) = \left(G(x,y) - G_{3^N \rho}(x,y)\right) + \sum_{i=1}^{N-1}\left(G_{3^{i+1}\rho}(x,y) -G_{3^i \rho}(x,y)\right) + G_{3\rho}(x,y).
\end{equation}
For $i=1, \ldots, N-1$ and for every $z \in \partial B_{3^i \rho}(y)$, we have
\[
G_{3^{i+1}\rho}(z,y) -G_{3^i \rho}(z,y) = G_{3^{i+1}\rho}(z, y) \ge C_1
\]
since we assumed the claim.

Note that $u:=G_{3^{i+1}\rho}(\,\cdot, y) -G_{3^i \rho}(\,\cdot,y)$ satisfies $L u =0$ in $B_{3^i \rho}(y)$.
By the comparison principle, we deduce that the previous inequality holds for all $z \in B_{3^i \rho}(y)$.
In particular, setting $z=x$ gives
\[
C_1 \le G_{3^{i+1}\rho}(x, y) -G_{3^i \rho}(x, y).
\]
A similar argument yields
\[
C_1 \le G(x, y) -G_{3^M \rho}(x, y).
\]
Therefore, using \eqref{mon1543} and \eqref{mon_240818}, we obtain
\[
C_2 \log\left(\frac{3r}{\abs{x-y}}\right) \le  C_1 (N+1) \le G(x, y). 
\]
This establishes the lower bound in \eqref{estimate_green} as well.

It remains only to prove the claim \eqref{lower_1229}.
Since the general case follows similarly, it suffices to consider $i=1$.
Consider the Green's function $G_{3 \rho}(\,\cdot, \cdot\,)$ for $L$ on $B_{3\rho}(y)$, and let $z \in \partial B_{\rho}(y)$.
Choose a nonnegative function $\eta \in C^{\infty}_c(B_{\rho}(z))$ such that
\[
\eta=1\;\text{ in }\;B_{3\rho/4}(z),\quad \norm{D\eta}_{L^\infty} \le 8/\rho,\quad \norm{D^2 \eta}_{L^\infty} \le 16/\rho^2.
\]

We employ the Harnack's inequality (\cite[Theorem 4.3]{GK24}) for nonnegative solutions of the double divergence form equation
\[
L^*u=D_{ij}(a^{ij}u)-D_i(b^i u)+cu=0.
\]
Using H\"older's inequality and the facts that $p_0>2$ and $c \le 0$, we obtain
\begin{align*}
1&=\eta(z)=\int_{B_{3 \rho}(y)} G_{3\rho}(z, \cdot\,) L \eta \leq \int_{B_{\rho}(z)\setminus B_{3\rho/4}(z)} G_{3\rho}^*(\,\cdot,z) \left(a^{ij}D_{ij}\eta + b^i D_i \eta\right)\\
&\le C \sup_{B_{\rho}\setminus B_{3\rho/4}} G_{3\rho}^*(\,\cdot,z) \left( \rho^{-2}\, \abs{B_{\rho}} + \rho^{-1} \,\norm{\vec b}_{L^{p_0}(B_{\rho}(z))} \abs{B_{\rho}}^{1-1/p_0}\right) \le C \sup_{B_{\rho}\setminus B_{3\rho/4}} G_{3\rho}^*(\,\cdot, z).
\end{align*}

Next, observe that any two points in $B_{\rho}(z)\setminus B_{3\rho/4}(z)$ can be connected by a chain of at most $\lceil 4\pi \rceil$ balls of radius $\rho/4$, all contained in $B_{5\rho/4}(z)\setminus B_{\rho/2}(z)$.

Applying Harnack's inequality iteratively to $G_{3\rho}^*(\,\cdot, z)$ on each balls, we deduce
\[
\sup_{B_{\rho}(z) \setminus B_{3\rho/4}(z)} G_{3\rho}^*(\,\cdot, z) \le C G_{3\rho}^*(y,z)= C G_{3\rho}(z,y).
\]
Thus, we conclude that
\[
G_{3\rho}(z,y) \ge C \quad \text{for all } z \in \partial B_{\rho}(y).
\]
Defining $C_1$ as the constant in the above inequality completes the proof.
\end{proof}

\begin{remark}			\label{rmk0807fri}
From the proof of Theorem \ref{thm_green_function}, it follows that 
\[
C_1 \log \left(\frac{5r}{2\abs{x-y}}\right) \le G(x,y)\le C_2 \log \left(\frac{4r}{\abs{x-y}}\right),\quad x \neq y \in \overline B_{3r/2},
\]
where $C_1$ and $C_2$ are positive constants depending only on $\lambda$, $\Lambda$, $\omega_{\mathbf A}$, and $\diam \mathcal B$.
We use this estimate in the proof of Theorem \ref{thm_wiener}.
\end{remark}

%------------------------------------------------------------------------------%
\section{Relative Potential and Capacity}		\label{sec3}
%------------------------------------------------------------------------------%
Throughout this section, we assume that
\[
c \equiv 0.
\]
We refer to \cite[Section 3.3]{DKK25} for definitions of $L$-supersolution, $L$-subsolution, etc.
We use the notation  $\overline H_f$ and $\underline H_f$ for the Perron upper and lower solutions, respectively, to the Dirichlet problem:
\[
Lu=0 \;\text{ in }\; \Omega,\qquad
u=f\; \text{ on }\;\partial \Omega.
\]
When $f$ is continuous, Wiener proved that for  $L=\Delta$,
\[
\overline H_f=\underline H_f.
\]
For a proof, see \cite[Theorem 3.6.16]{Helms}, which also applies to general operators $L$ as shown below.

\begin{lemma}
Let $\Omega \subset \bR^2$ be a bounded open set and $f \in C(\partial\Omega)$.
Then $\overline H_f=\underline H_f$. 
\end{lemma}
\begin{proof}
Let $p_n$ be a sequence of polynomials such that $p_n \to f$ uniformly on $\partial\Omega$.
Consider a ball $\mathcal{B}$ containing $\overline\Omega$, and let $G(x,y)$ be the Green's function for $L$ in $\mathcal{B}$.
Noting that $Lp_n \in L^p(\mathcal B)$ for some $p>1$, define
\[
v_n(x)=\int_{\mathcal B} G(x,y)\, \abs{L p_n(y)}\,dy.
\]
Clearly, $v_n$ is a continuous $L$-supersolution, and $p_n+v_n$ is also a continuous $L$-supersolution.
Thus, $u$ can be approximated uniformly on $\partial\Omega$ by the difference of two continuous $L$-supersolutions.
The remainder of the proof follows exactly as in \cite[Theorem 3.6.16]{Helms}.
\end{proof}

Thus, for continuous $f$, we will use $H_f$ to denote the Perron solution.

We recall that a point $x_0 \in \partial \Omega$ is called a regular point if for all $f \in C(\partial \Omega)$, the Perron solution $H_f$ satisfies
\[
\lim_{x \to x_0, \, x\in \Omega}\, H_f(x)=f(x_0).
\]

It is well known that a point $x_0$ is a regular point if and only if there exists a barrier at $x_0$.
A function $w$ is called a \textit{barrier} (with respect to $\Omega$) at $x_0$ if:
\begin{enumerate}[leftmargin=*]
\item[(i)]
$w$ is an $L$-supersolution in $\Omega$.
\item[(ii)]
For any $\delta>0$, there exists $\epsilon>0$ such that $w \ge \epsilon$ on $\partial\Omega \setminus B_\delta(x_0)$.
\item[(iii)]
$\lim_{x\to x_0,\,x \in \Omega} w(x)=0$.
\end{enumerate}

The following result establishes that being a regular point is a local property:
\begin{lemma}		\label{local_property_barrier}
A point $x_0 \in \partial \Omega$ is a regular point with respect to $\Omega$ if and only if $x_0$ is a regular point with respect to $B_r(x_0)\cap \Omega$ for some $r>0$.
\end{lemma}

\begin{proof}
Suppose $x_0$ is regular with respect to $\Omega$, so that there exists a barrier $w$ at $x_0$ with respect to $\Omega$.
Then, the restriction of $w$ to $B_r(x_0)\cap \Omega$ is clearly a barrier at $x_0$ with respect to $B_r(x_0) \cap \Omega$. 
    
Conversely, suppose there exists a barrier $w$ at $x_0$ with respect to $B_r(x_0) \cap \Omega$.
Since $w$ is lower semicontinuous, it is positive on $\partial B_{r/2}(x_0) \cap \Omega$.
Define
\[
m:= \min_{\partial B_{r/2} \cap \Omega} w>0.
\]
We now define a function $\tilde w$ by
\[
\tilde w:=\begin{cases} \min (w, m) &\text{ in }\; B_{r/2}(x_0)\cap \Omega,\\
m&\text { in } \Omega\; \setminus B_{r/2}(x_0).
\end{cases}
\]
Then $\tilde w$ is a barrier at $x_0$ with respect to $\Omega$.
\end{proof}

\medskip

In \cite{DKK25}, we fixed a large ball $\mathcal{B}$ and for a set $E \Subset \mathcal{B}$, we introduced $\hat u_E$, the capacitary potential of $E$.
In particular, $\hat u_E$ vanishes on $\partial \mathcal{B}$.
In the current two-dimensional setting, it is convenient to consider a capacitary potential of $E$ relative to $\mathcal{D}$, where $\mathcal{D} \subset \mathcal{B}$ is not fixed.
In our main application, $\mathcal{D}$ will also be a ball.
We denote by $G_{\mathcal{D}}(x,y)$ the Green's function for $L$ in $\mathcal{D}$.

\begin{definition}			\label{def_rfn}
For $E \Subset \mathcal{D}$, we define
\[
u_{E,  \mathcal{D}}(x):=\inf \,\Set{v(x):  v \in \mathfrak{S}^+(\mathcal{D}),\; v\ge 0\text{ in }\mathcal{D},\; v \ge 1 \text{ in }E},\quad x \in \mathcal{D}.
\]
The lower semicontinuous regularization $\hat u_{E,\mathcal{D}}$, defined by
\[
\hat u_{E,\mathcal{D}}(x)=\sup_{r>0} \left( \inf_{\mathcal{D} \cap B_r(x)} u_{E,\mathcal{D}}\right),
\]
is called the capacitary potential of $E$ relative to $\mathcal{D}$.
\end{definition}

\begin{lemma}				\label{lem1013sat} 
The following are true:
\begin{enumerate}[leftmargin=*, label=(\alph*)]
\item		\label{item_a}
$0\le \hat u_{E,\mathcal{D}} \le u_{E,\mathcal{D}} \le 1$.
\item
$u_{E,\mathcal{D}}=1$ in $E$ and $\hat u_{E,\mathcal{D}} =1$ in $\intr (E)$.
\item
$\hat u_{E,\mathcal{D}}$ is a potential in $\mathcal{D}$.
That is, it is a nonnegative $L$-supersolution in $\mathcal{D}$, which finite at every point in $\mathcal{D}$, and vanishes continuously on $\partial \mathcal{D}$.
\item
$u_{E,\mathcal{D}}=\hat u_{E,\mathcal{D}}$ in $\mathcal{D}\setminus \overline{E}$ and
$L u_{E,\mathcal{D}}=L \hat u_{E,\mathcal{D}}=0$ in $\mathcal{D}\setminus \overline{E}$.
\item		\label{item_e}
If $x_0\in \partial E$, we have
\[
\liminf_{x\to x_0} \hat u_{E,\mathcal{D}}(x) =\liminf_{x\to x_0,\, x\in \mathcal{D} \setminus  E} \hat u_{E,\mathcal{D}}(x).
\]
\item		\label{item_f}
$\hat u_{E_1,\mathcal{D}} \le \hat u_{E_2,\mathcal{D}}$\, whenever $E_1 \subset E_2 \subset \mathcal{D}$.
\item		\label{item_g}
$\hat u_{E,\mathcal{D}_1} \le \hat u_{E,\mathcal{D}_2}$\, whenever $E \subset \mathcal{D}_1\subset \mathcal{D}_2$.
\end{enumerate}
\end{lemma} 
\begin{proof}
For \ref{item_a}--\ref{item_e}, refer to the proof of \cite[Lemma 5.6]{DKK25}.
Properties \ref{item_f} and \ref{item_g} follow directly from Definition \ref{def_rfn} by the comparison principle.
\end{proof}

The following lemma gives a convenient characterization of a regular point in terms of relative capacitary potentials.
\begin{lemma}		\label{lem_charact}
Let $\Omega \Subset \mathcal{B}$ and let $x_0 \in \partial \Omega$.
Denote
\[
B_r=B_r(x_0), \quad E_r = \overline{B_r}\setminus\Omega.
\]
Then, $x_0$ is regular if and only if
\[
\hat{u}_{E_r,B_{4r}}(x_0)=1
\]
for every $r$ satisfying $0<r<r_0:=\tfrac{1}{4}\dist(x_0,\partial\mathcal{B})$.
\end{lemma}
\begin{proof}
\noindent
\textbf{(Necessity)}
Suppose that $\hat{u}_{E_r,B_{4r}}(x_0)<1$ for some $r \in (0,r_0)$.
Define the function
\[
f(x)=\left(1-\abs{x-x_0}/r\right)_+.
\]
Let $u$ be the lower Perron solution to the Dirichlet problem
\[
Lu=0\;\text{ in }\;\Omega\cap B_{4r},\quad  u=f\;\text{ on }\;\partial(\Omega\cap B_{4r}).
\]
By Definition \ref{def_rfn} and Lemma \ref{lem1013sat}(d), we have
\[
u \le u_{E_r, B_{4r}}=\hat u_{E_r, B_{4r}}\;\text{ in }\;\Omega\cap B_{4r}.
\]
This contradicts the regularity at $x_0$, since
\[
\liminf_{x\to x_0,\,x\in \Omega\cap B_{4r}}\ u(x) \le  \liminf_{x\to x_0}\, \hat u_{E_r, B_{4r}}(x)= \hat u_{E_r, B_{4r}}(x_0) < 1=f(x_0),
\]
where we applied Lemma \ref{lem1013sat}(e) and the fact that
\[
\liminf_{x \to x_0} \hat u_{E_r, B_{4r}}(x)=\hat u_{E_r, B_{4r}}(x_0),
\]
which is due to $u$ is an $L$-supersolution in $B_{4r}$ (see \cite[Lemma 3.21(c)]{DKK25}).
Thus, by Lemma \ref{local_property_barrier}, the point $x_0$ is not regular.
 
\medskip
\noindent
\textbf{(Sufficiency)}
Suppose that $\hat{u}_{E_r,B_{4r}}(x_0)=1$ for every $r\in(0,r_0)$.
Observe that
\[
1= \hat{u}_{E_r,B_{4r}}(x_0) \leq \hat{u}_{E_r,\mathcal{B}}(x_0) \leq 1,
\]
which implies that $\hat{u}_{E_r,\mathcal{B}}(x_0)=1$ for all $r \in (0,r_0)$.

Let $f$ be a continuous function on $\partial \Omega$.
Then, for every $\epsilon>0$, there exists $r \in (0,r_0)$ such that $\abs{f(x) - f(x_0)}< \epsilon$ for all $x \in \partial \Omega$ satisfying $\abs{x-x_0}<2r$. 

On the other hand, since
\[
L\hat{u}_{E_r,\mathcal{B}}=0\;\text{ in }\;\mathcal{B} \setminus \overline B_{3r/2},
\]
the strong maximum principle implies the existence of $\delta \in (0,1)$ such that
\[
\sup_{\partial B_{2r}} \hat{u}_{E_r,\mathcal{B}}=1-\delta.
\]
Define $M:=\sup_{\partial \Omega} \,\abs{f-f(x_0)}$ and consider the function
\[
v:=f(x_0)+\epsilon + \frac{M}{\delta}(1- \hat{u}_{E_r,\mathcal{B}}).
\]
Note that $v \ge f(x_0)+\epsilon \ge f$ on $\partial \Omega \cap B_{2r}(x_0)$, and $v \ge f(x_0)+M \ge  f$ on $\partial\Omega \setminus B_{2r}(x_0)$, since the maximum principle ensures that $\hat{u}_{E_r,\mathcal{B}} \le 1-\delta$ in $\mathcal{B}\setminus B_{2r}(x_0)$.
Thus, we conclude that $v \ge f$ on $\partial \Omega$.
Let $\overline H_f$ be the upper Perron solution of the problem
\[
Lu=0\;\text{ in }\;\Omega,\quad  u=f\;\text{ on }\;\partial \Omega.
\]
It is clear that $\overline H_f \le v$ in $\Omega$.
Thus, we obtain
\[
\limsup_{x \to x_0,\, x \in \Omega} \overline H_f(x) \le \limsup_{x \to x_0,\, x \in \Omega}  v(x) = f(x_0) + \epsilon + \frac{M}{\delta}(1-\liminf_{x \to x_0} \hat{u}_{E_r,\mathcal{B}}(x))=f(x_0) + \epsilon,
\]
where we used Lemma \ref{lem1013sat}(e).
Similarly, consider 
\[
v:=f(x_0)-\epsilon + \frac{M}{\delta}(\hat{u}_{E_r,\mathcal{B}}-1).
\]
Observing that $v \le f$ on $\partial \Omega$, we obtain
\[
\liminf_{x \to x_0,\, x \in \Omega} \underline H_f(x) \ge \liminf_{x \to x_0,\, x \in \Omega} v(x) = f(x_0)-\epsilon+\frac{M}{\delta}(\liminf_{x \to x_0} \hat{u}_{E_r,\mathcal{B}}(x)-1)=f(x_0)-\epsilon.
\]
Since $\epsilon>0$ is arbitrary, we conclude that $x_0$ is a regular point.
\end{proof}

\begin{theorem}			\label{thm1929}
Let $K$ be a compact subset of $\mathcal{D}$.
There exists a Borel measure $\mu_{K, \mathcal{D}}$, called the capacitary measure of $K$ relative to $\mathcal{D}$, which is supported in $\partial K$ and satisfies
\[
\hat u_{K,\mathcal{D}}(x)=\int_{K} G_{\mathcal{D}}(x, y)\,d\mu_{K, \mathcal{D}} (y), \quad \forall x \in \mathcal{D}.
\]
\end{theorem}
\begin{proof}
Refer to the proof of \cite[Theorem 5.14]{DKK25}.
\end{proof}

\begin{definition}
Let $K$ be a compact subset of $\mathcal{D}$.
The capacity of $K$ relative to $\mathcal{D}$ is defined as
\[
\capacity(K,\mathcal{D})=\mu_{K,\mathcal{D}}(K).
\]
\end{definition}

\begin{lemma}		\label{lem2020sat}
Let $K$ be a compact subset of $\mathcal{D}$ and let $\mathfrak{M}_{K,\mathcal{D}}$ be the set of all Borel measures $\nu$ supported in $K$ satisfying $\int_K G_{\mathcal{D}}(x,y)\,d\nu(y) \le 1$ for every $x \in \mathcal{D}$.
Then
\[
\hat u_{K,\mathcal{D}}(x)=\sup_{\nu \in \mathfrak{M}_{K,\mathcal{D}}} \int_K G_{\mathcal{D}}(x,y)\,d\nu(y).
\]
\end{lemma}
\begin{proof}
Refer to the proof of \cite[Lemma 5.22]{DKK25}.
\end{proof}

We primarily consider the case where $\mathcal{D}=B_{4r}(x_0)\subset \mathcal{B}$ and $K \subset \overline B_r(x_0)$.

\begin{lemma}				\label{lem0956fri}
Let $B_{4r}=B_{4r}(x_0) \subset \mathcal{B}$.
For any compact set $K \subset \overline B_r$ and any point $x^o \in \partial B_{3r/2}$, we have
\[
\frac{1}{C_0 \log 6}\; \hat u_{K,B_{4r}}(x^o) \le \capacity(K,B_{4r}) \le \frac{C_0}{\log (6/5)}\; \hat u_{K,B_{4r}}(x^o),
\]
where $C_0$ is the constant from Theorem \ref{thm_green_function}.
\end{lemma}
\begin{proof}
Let $G(x,y)$ be the Green's function for $L$ in $B_{4r}$.
Applying Theorem \ref{thm1929} and Theorem \ref{thm_green_function}, we obtain
\begin{align*}
\hat u_{K, B_{4r}}(x^o) &\le C_0 \int_K \log \left(\frac{3r}{\abs{x^o-y}}\right)\,d\mu_{K, B_{4r}}(y) \le C_0 (\log 6) \capacity(K, B_{4r}),\\
\hat u_{K, B_{4r}}(x^o) &\ge \frac{1}{C_0} \int_K \log \left(\frac{3r}{\abs{x^o-y}}\right)\,d\mu_{K, B_{4r}}(y) \ge \frac{ \log(6/5)}{C_0} \capacity(K, B_{4r}).
\end{align*}
The desired conclusion follows immediately from these estimates.
\end{proof}

\begin{theorem}			\label{thm1939}
Let $B_{4r}=B_{4r}(x_0) \subset \mathcal{B}$.
Let $K \subset \overline B_r$ be a compact set.
Denote by $\capacity^{\Delta}(K,B_{4r})$ the capacity of $K$ relative to $B_{4r}$ associated with the Laplace operator.
There exists a constant $C>0$, depending only on $\lambda$, $\Lambda$, $\omega_{\mathbf A}$, and $\diam \mathcal{B}$, such that
\[
C^{-1} \capacity^{\Delta}(K, B_{4r}) \le \capacity(K, B_{4r}) \le C \capacity^{\Delta}(K, B_{4r}).
\]
\end{theorem}
\begin{proof}
The proof follows the approach of \cite[Theorem 5.24]{DKK25}, but with the Green's function bounds from Theorem \ref{thm_green_function} in place of those used there.
\end{proof}

\begin{proposition}[Capacity of a ball]	\label{capa_B_r}
Let $B_{4r}=B_{4r}(x_0) \subset \mathcal{B}$.
There exist positive constants $C_1$ and $C_2$ such that for every ball $B_s(y_0) \subset B_r$, we have
\begin{equation}\label{mon1114} 
\frac{C_1}{\log (3r/s)} \le \capacity\left(\overline B_s(y_0),B_{4r}\right) \le \frac{C_2}{\log (3r/s)} .
\end{equation}
\end{proposition}
\begin{proof}
Let $G(x,y)$ be the Green's function for $L$ in $B_{4r}$.
Define $M$ and $m$ as
\begin{align*}
M:= \sup\, \{G(x, y_0) : x \in \partial B_s(y_0)\},\\
m:= \inf\, \{G(x, y_0) : x \in \partial B_s(y_0)\}.
\end{align*}
Next, we define two functions $v$ and $w$ by
\[
v(x)= \min \left(\frac{1}{m} G(x, y_0),1\right),\quad
w(x)= \min \left(\frac{1}{M} G(x, y_0),1\right).
\]
Note that $v$ is nonnegative and satisfies $v=1$ on $\partial B_s(y_0)$.
Since $v$ is an $L$-supersolution in $B_{4r}$, the strong minimum principle implies that $v \ge 1$ in $\overline B_s(y_0)$.
Consequently, we obtain
\begin{equation}			\label{eq1840tue}
\hat u_{\overline B_s(y_0), B_{4r}}  \le v \le \frac{1}{m} G(\,\cdot,y_0)  \; \text{ in } B_{4r}.
\end{equation}

On the other hand, since $G(\,\cdot, y_0) \le M$ on $\partial B_s(y_0)$ and $G(\,\cdot, y_0)=0$ on $\partial B_{4r}$, the maximum principle implies that $G(\,\cdot, y_0) \le M$ in $B_{4r} \setminus \overline B_s(y_0)$.
Thus, we obtain
\begin{equation}			\label{eq1927tue}
w=\frac{1}{M}G(\,\cdot, y_0) \le 1 \;\text{ in }\;B_{4r} \setminus B_s(y_0).
\end{equation}
Therefore, by the comparison principle (see Lemma \ref{lem1013sat}), we obtain
\[
w \le u_{\overline B_t(y_0),B_{4r}}\;\text{ in }\;B_{4r} \setminus \overline B_s(y_0).
\]
Moreover, it follows directly from Definition \ref{def_rfn} that
\[
w \le u_{\overline B_s(y_0),B_{4r}}\;\text{ in }\;\overline B_s(y_0).
\]
Consequently, we conclude that (see \cite[Lemma 5.4(c)]{DKK25})
\[
w \le \hat u_{\overline B_s(y_0),B_{4r}} \; \text{ in } B_{4r}.
\]
In particular, from \eqref{eq1927tue}, we obtain
\begin{equation}			\label{eq1841tue}
\frac{1}{M}G(\,\cdot, y_0)  \le \hat u_{\overline B_s(y_0), B_{4r}} \; \text{ in } B_{4r}\setminus B_s(y_0).
\end{equation}
By choosing $x=x^o \in \partial B_{3r/2}(y_0)$ in \eqref{eq1840tue} and \eqref{eq1841tue}, and applying the Green's function estimates from Theorem \ref{thm_green_function}, we obtain
\begin{equation}	\label{mon1100}
\frac{\log 2}{C_0^2 \log (3r/s)} \le \frac{1}{M}G(x^o, y_0) \le \hat u_{\overline B_s(y_0),B_{4r}}(x^o) \le  \frac{1}{m}G(x^o, y_0) \le \frac{C_0^2 \log 2}{\log (3r/s)}.
\end{equation}
Finally, combining \eqref{mon1100} with Lemma \ref{lem0956fri}, we obtain \eqref{mon1114}.
\end{proof}

%-------------------------------------------------------%
\section{Wiener test and applications}			\label{sec4}
%-------------------------------------------------------%

The following theorem establishes the Wiener criterion for a regular point.

\begin{theorem}		\label{thm_wiener}
Assume that Conditions \ref{cond1} and \ref{cond2} hold with $c\equiv 0$.
Let $\Omega \Subset \mathcal{B}$ be a domain, and consider a point $x_0 \in \partial \Omega$. Define $r_0=\frac{1}{8}\dist(x_0, \partial \mathcal{B})$.
Then $x_0$ is a regular point if and only if
\[
\sum_{k=0}^\infty \capacity \left(\,\overline{B_{2^{-k}r_0}(x_0)}\setminus \Omega,B_{2^{2-k}r_0}(x_0)\right)=\infty.
\]
\end{theorem}

\begin{proof}
We introduce the following notation:
\[
B_r=B_r(x_0),\quad
E_r=\overline B_r\setminus \Omega, \quad  u_r= u_{E_r, B_{4r}},\quad \mu_r=\mu_{E_r, B_{4r}}.
\]
Additionally, we use the notation $G_r(x,y)$ for the Green's function for $L$ on $B_r$.

\medskip
\noindent
\textbf{(Necessity)}
We will show that $x_0$ is not a regular point if
\begin{equation}			\label{eq0801thu}
\sum_{k=0}^\infty \capacity \left(E_{2^{-k}r_0}, B_{2^{2-k}r_0}\right)<\infty.
\end{equation}
By Theorem \ref{thm1929}, we have
\begin{equation}			\label{eq1353thu}
\hat u_r(x)=\int_{E_r} G_{4r}(x, y)\,d\mu_r(y).
\end{equation}
For $0<\rho<r$, it follows that
\[
\int_{E_\rho}G_{4r}(x_0, y)\,d\mu_r(y)  \le \int_{E_r}G_{4r}(x_0, y)\,d\mu_r(y)=\hat u_{r}(x_0) \le 1.
\]
Applying the dominated convergence theorem, we obtain
\[
\lim_{\rho \to 0} \int_{E_\rho}G_{4r}(x_0, y)\,d\mu_r(y) = \int_{\set{x_0}} G_{4r}(x_0,y)\,d\mu_r(y) \le 1.
\]
Since $G_{4r}(x_0,x_0)=\infty$, we must have $\mu_r(\set{x_0})=0$.
Consequently, we obtain
\begin{equation}			\label{eq1445thu}
\int_{\set{x_0}} G_{4r}(x_0,y)\,d\mu_r(y)=0.
\end{equation}
Next, we partition $E_r$ into annular regions and use the Green's function estimates from Theorem~\ref{thm_green_function}.
From \eqref{eq1353thu} and \eqref{eq1445thu}, we derive
\begin{align*}
\hat u_{r}(x_0) &= \sum_{k=0}^{\infty} \int_{E_{2^{-k}r}\setminus E_{2^{-k-1}r}} G_{4r}(x_0, y)\,d\mu_r(y)+ \int_{\set{x_0}} G_{4r}(x_0,y)\,d\mu_r(y) \\
& \lesssim \sum_{k=0}^\infty \log \left(\frac{3r}{2^{-k-1}r}\right) \mu_r(E_{2^{-k}r}\setminus E_{2^{-k-1}r}) \simeq \sum_{k=0}^\infty (k+1) \mu_r(E_{2^{-k}r}\setminus E_{2^{-k-1}r}).
\end{align*}
By an elementary computation, we observe that
\[
\sum_{k=0}^N (k+1)\mu_r(E_{2^{-k}r}\setminus E_{2^{-k-1}r})= \sum_{k=0}^N \mu_r(E_{2^{-k}r}\setminus E_{2^{-N-1}r}).
\]
Consequently, we obtain
\[
\hat u_r(x_0) \lesssim \lim_{N\to\infty}  \sum_{k=0}^N (k+1)\mu_r(E_{2^{-k}r}\setminus E_{2^{-k-1}r}) \lesssim \lim_{N\to\infty} \sum_{k=0}^N \mu_r(E_{2^{-k}r} \setminus E_{2^{-N-1}r}).
\]
Since the measures are non-negative and increasing, we conclude
\begin{equation}			\label{eq1402tue}
\hat u_r(x_0) \lesssim \lim_{N\to\infty} \sum_{k=0}^N \mu_r(E_{2^{-k}r}) = \sum_{k=0}^{\infty} \mu_r(E_{2^{-k}r}).
\end{equation}

We will show that there exists a constant $C$ such that
\begin{equation}			\label{eq1727sat}
\mu_r(E_t) \le C \mu_t(E_t)=C \capacity(E_t,B_{4t}), \quad \forall t \in (0,r].
\end{equation}
Assuming \eqref{eq1727sat} holds and setting $r=2^{-n} r_0$ in \eqref{eq1402tue}, we obtain
\[
\hat u_{2^{-n} r_0}(x_0) \lesssim \sum_{k=0}^\infty \capacity(E_{2^{-k-n} r_0},B_{2^{2-k-n} r_0}) = \sum_{k=n}^\infty \capacity(E_{2^{-k} r_0},B_{2^{2-k} r_0}).
\]
By hypothesis \eqref{eq0801thu}, this sum converges to zero as $n \to \infty$.
Choose $n$ sufficiently large so that $\hat u_{2^{-n} r_0}(x_0)<1$.
Lemma~\ref{lem_charact} implies that $x_0$ is not a regular point.

It remains to prove \eqref{eq1727sat}.
For $0<t \le r$, define the measure $\nu$ by
\[
\nu(E)=\mu_r(E_t \cap E).
\]
Clearly, $\nu$ is supported in $E_t$.
Furthermore, for every $x \in B_{4t}$, we have
\begin{align*} 
\int_{E_t} G_{4t}(x,y)\,d\nu(y) &=\int_{E_t} G_{4t}(x,y)\,d\mu_r(y) \le \int_{E_t} G_{4r}(x,y)\,d\mu_r(y)\\ &\le \int_{E_r} G_{4r}(x,y)\,d\mu_r(y) \le 1.
\end{align*}
Here, we used the fact that $G_{4r}(x,y)\geq G_{4t}(x,y)$ for all $x, y \in B_{4t}$.
By applying Lemma~\ref{lem2020sat}, we deduce that
\[
\int_{E_t} G_{4t}(x,y)\,d\mu_r(y)=\int_{E_t} G_{4t}(x,y)\,d\nu(y) \le \hat u_{t}(x)=\int_{E_t} G_{4t}(x,y)\,d\mu_t(y),\quad \forall x \in B_{4t}.
\]
Taking $x$ on $\partial B_{3t/2}$ in the inequality above and using Remark \ref{rmk0807fri}, we obtain \eqref{eq1727sat}.
This completes the proof for necessity.

\medskip
\noindent
\textbf{(Sufficiency)}
We state the following lemma, which is similar to \cite[Lemma 6.8]{DKK25}. However, the proof is adjusted to the two-dimensional setting.

\begin{lemma}			\label{lem2119thu}
There exist a constant $\gamma>0$ such that for all $t \in (0, r)$, we have
\[
\sup_{\Omega\cap B_t(x_0)}\left(1- \hat u_r \right)  \le \exp\left\{-\gamma\sum_{k=0}^{\lfloor \log_2(r/t) \rfloor}\capacity(E_{2^{-k}r},B_{2^{2-k}r})\right\}.
\]
\end{lemma}
\begin{proof}
To investigate the local behavior of $\hat u_{r}$, we construct a sequence of auxiliary functions $\{v_i\}$ as follows.
We begin by defining
\[
v_0 :=  u_r \quad\text{and}\quad m_0 := \inf_{B_r} \, \hat v_0=\inf_{B_r} \,\hat u_r.
\]
Then, for $i=1,2,\ldots$, we define
\[
v_i := m_{i-1} + (1-m_{i-1}) u_{2^{-i}r}\quad\text{and}\quad m_i := \inf_{B_{2^{-i}r}} \,\hat v_i=m_{i-1}+(1-m_{i-1}) \inf_{B_{2^{-i}r}} \,\hat u_{2^{-i}r}.
\] 

Using Theorem \ref{thm1929}, for all $x \in B_{2^{-i}r}$, we have
\[
\hat u_{2^{-i}r}(x)  \geq \left(\inf_{y \in B_{2^{-i}r}}G_{2^{2-i}r}(x, y)\right)\, \capacity(E_{2^{-i}r},B_{2^{2-i}r})  \ge \gamma \capacity(E_{2^{-i}r},B_{2^{2-i}r}),
\]
where $\gamma=\log(3/2)/C_0>0$ by Theorem \ref{thm_green_function}.
We observe that
\[
1-\hat v_i = (1-m_{i-1})\left(1-\hat u_{2^{-i}r}\right).
\]
Taking the supremum over $B_{2^{-i}r}$ and using the inequality $1-y \leq e^{-y}$, we obtain
\[
1-m_i \leq (1-m_{i-1})(1-\gamma \capacity(E_{2^{-i}r},B_{2^{1-i}r})) \leq (1-m_{i-1})\exp\left\{-\gamma \capacity(E_{2^{-i}r},B_{2^{2-i}r})\right\}.
\]
By iterating, we deduce that
\begin{equation}		\label{eq2042tue}
\sup_{B_{2^{-i}r}}\, (1-\hat v_i) = 1-m_{i}\, \le\, \exp\left\{-\gamma \sum_{k=0}^{i}\capacity(E_{2^{-k}r},B_{2^{2-k}r})\right\}.
\end{equation}
We will now show that
\begin{equation}	\label{eq2043tue}
\hat v_j  \le \hat u_{r} \; \text{ in }\; B_{2^{2-j}r},\quad j=1,2,\ldots.
\end{equation}

Assuming \eqref{eq2043tue} for now, let $n=\lfloor \log_2(r/t) \rfloor$ so that $2^{-n-1}r < t \le 2^{-n}r$.
Then, it follows from \eqref{eq2042tue} and \eqref{eq2043tue} that
\[
\sup_{\Omega\cap B_t}\left(1- \hat u_{r}\right)  \leq \, \sup_{B_{2^{-n}r}}\,(1-\hat v_n)  \le  \exp\left\{-\gamma\sum_{k=0}^{\lfloor \log_2(r/t) \rfloor}\capacity(E_{2^{-k}r},B_{2^{2-k}r})\right\}.
\]

To prove \eqref{eq2043tue}, it is sufficient to show that $\hat v_i \ge \hat v_{i+1}$ in $B_{2^{1-i}r}$, for $i=0,1,2,\ldots$, since $\hat v_0= \hat u_r$.
We note that $v_{i+1}$ is characterized as follows:
\[
v_i(x)=\inf_{w \in \mathscr{F}_i} w(x),\quad x \in B_{2^{2-i}r},
\]
where
\[
\mathscr{F}_i=\left\{ w \in \mathfrak{S}^+(B_{2^{2-i}r}):  w\ge m_i\;\text{ in }\;B_{2^{2-i}r},\;\; w \ge 1\; \text{ in }\;E_{2^{-i}r}\right\}.
\]

It then follows that $\hat v_{i+1}$ satisfies $L \hat v_{i+1} =0$ in $B_{2^{1-i}r}\setminus E_{2^{-i-1}r}$ (see Lemma \ref{lem1013sat}(d)).
Moreover, for $w \in \mathscr{F}_i$, we make the following observations:
\begin{enumerate}[label=(\roman*)]
\item
$\hat v_{i+1}(x) \le 1 \le w(x)$ for all $x \in E_{2^{-i-1}r}$.
\item
For $x \in \partial E_{2^{-i-1}r}$, we have
\[
\limsup_{B_{2^{1-i}r} \setminus E_{2^{-i-1}r}\, \ni y \to x} \hat v_{i+1}(y) \le 1 \le w(x) \le \liminf_{B_{2^{1-i}r} \setminus E_{2^{-i-1}r}\,\ni y \to x} w(y).
\]
\item
Suppose $x \in \partial B_{2^{1-i}r}$.
Since $\hat u_{2^{-i-1}r}$ vanishes on $\partial B_{2^{1-i}r}$, it follows from the identity $\hat v_{i+1}- m_i = (1-m_i)\hat u_{2^{-i-1}r}$ that
\[
\limsup_{B_{2^{1-i}r} \setminus E_{2^{-i-1}r}\, \ni y \to x} \hat v_{i+1}(y) = m_i \le \liminf_{B_{2^{1-i}r} \setminus E_{2^{-i-1}r}\, \ni y \to x} w(y). 
\]
\end{enumerate}

Thus, the comparison principle implies that $\hat v_{i+1} \le w$  in $B_{2^{1-i}r} \setminus E_{2^{-i-1}r}$.
Since $\hat v_{i+1} \le w$ also holds on $E_{2^{-i-1}r}$, we have $\hat v_{i+1} \le w$ throughout $B_{2^{1-i}r}$.
By taking the infimum over $w \in \mathscr F_i$ and applying lower semicontinuous regularization, we obtain
\[
\hat v_{i+1}\le \hat v_i \;\text{ in }\;B_{2^{1-i}r}.
\]
Therefore, \eqref{eq2043tue} is proven, and the proof of the lemma is complete.
\end{proof}

Lemma~\ref{lem2119thu} establishes that $\liminf_{x \rightarrow x_0} \hat u_{r}(x)=1$ for all sufficiently small $r>0$.
Consequently, by Lemma~\ref{lem_charact}, $x_0$ is a regular point.

This concludes the proof for necessity.
The proof is now complete.
\qedhere
\end{proof}

The following theorems are immediate consequences of the Wiener's test.

\begin{theorem}		\label{thm0800sat}
Assume Conditions \ref{cond1} and \ref{cond2}, and let $\Omega$ be a bounded open domain in $\mathbb{R}^2$.
A point $x_0 \in \partial \Omega$ is a regular point for $L$ if and only if $x_0$ is a regular point for the Laplace operator.
\end{theorem}
\begin{proof}
Refer to the proof of \cite[Theorem 6.14]{DKK25}.
\end{proof}

A domain $\Omega$ is called a \textit{regular domain} if every point on its boundary $\partial \Omega$ is a regular point with respect to the Laplace operator.

\begin{theorem}		\label{thm0802sat}
Assume that Conditions \ref{cond1} and \ref{cond2} hold.
Let $\Omega$ be a bounded regular domain in $\mathbb{R}^2$.
For $f \in C(\partial\Omega)$, the Dirichlet problem,
\[
Lu=0 \;\text{ in }\;\Omega, \quad u=f \;\text{ on }\;\partial\Omega,
\]
has a unique solution $u \in W^{2, p_0/2}_{\rm loc}(\Omega) \cap C(\overline \Omega)$.
\end{theorem}

\begin{proof}
Refer to the proof of \cite[Theorem 6.16]{DKK25}.
\end{proof}

We construct the Green's function for $L$ in any regular two-dimensional domain.

\begin{theorem}			\label{thm1127sat}
Under Conditions \ref{cond1} and \ref{cond2}, let $\Omega \subset \bR^2$ be a bounded regular domain contained in $\mathcal{B}$.
Then, there exists a Green's function $G(x,y)$ in $\Omega$, which satisfies the following pointwise bound:
\[
0\le G(x,y) \le C \left\{1+ \log \left(\frac{\diam \Omega}{\abs{x-y}}\right)\right\},\quad x\neq y \in \Omega.
\]
where $C$ is a constant depending only on $d$, $\lambda$, $\Lambda$, $p_0$, $\diam \mathcal{B}$.
\end{theorem}
\begin{proof}
We choose $B_r=B_r(x_0) \subset \mathcal{B}$ such that $\Omega \subset B_r$ and $2r=\diam \Omega$.
Following the same proof as in \cite[Theorem 7.3]{DKK25} and utilizing Theorem \ref{thm2.3}, we construct the Green's function in $\Omega$.
This yields the desired upper bound as well.
\end{proof}

For additional comments on the Green's function, see Remark \ref{rmk2001mon}.
The following lemma provides a quantitative estimate for the oscillation of a solution to the Dirichlet problem in terms of the capacity at a point $x_0\in \partial\Omega$.

\begin{lemma}		\label{lem_bdyest}
Assume that Conditions \ref{cond1} and \ref{cond2} hold with $c\equiv 0$.
For $f \in C(\partial \Omega)$, let $u$ be the Perron solution to the Dirichlet problem
\[
Lu=0 \;\text{ in }\;\Omega, \quad u=f \;\text{ on }\;\partial\Omega.
\]
For a point $x_0\in \partial \Omega$, set $B_r=B_r(x_0)$.
Then, for all $t \in (0,r)$, the following estimate holds:
\begin{equation}	\label{bdry_osc}
\osc_{\Omega \cap B_t} u \le \osc_{\partial \Omega \cap B_{4r}} f + 2\left(\sup_{\partial \Omega} \,\abs{f}\right)\exp\left\{-\gamma\sum_{k=0}^{\lfloor \log_2(r/t) \rfloor}\capacity\left(\overline B_{2^{-k}r}\setminus \Omega, B_{2^{2-k}r}\right)\right\}.
\end{equation}
\end{lemma} 

\begin{proof}
We recall from \cite{DKK25} that lower Perron solution $\underline H_f$ is defined by
\[
\underline H_f(x)=\sup_{w \in \mathfrak{L}_f} w(x),
\]
where
\[
\mathfrak{L}_f=\left\{w : w \text{ is an $L$-subsolution in $\Omega$,}\; \;\limsup_{y\to x, \,y \in \Omega}\, w(y) \le f(x),\;\forall x\in \partial\Omega\right\}.
\]

Define the function $v$ by
\[
v:= \left(\sup_{\partial \Omega}f - \sup_{\partial \Omega \cap B_{4r}} f\right) (1-\hat u_{E_{r}, B_{4r}}) + \sup_{\partial \Omega \cap B_{4r}}f.
\]
Clearly, $v$ is an $L$-subsolution in $B_{4r}$.

For any $w \in \mathfrak{L}_f$, we verify that $w \leq v$ on $\partial(\Omega \cap B_{4r})$ as follows:
\begin{enumerate}[label=(\roman*)]
\item
Boundary on $\partial \Omega \cap B_{4r}$:
Since $w \in \mathfrak{L}_f$, we have
\[
\limsup_{\Omega\cap B_{4r} \ni y \to x}w(y) \,\leq\, \sup_{\partial \Omega \cap B_{4r}}f \,\leq\, v(x)\quad\text{for }\;x \in \partial \Omega \cap B_{4r}.
\]
\item
Boundary on $\Omega \cap \partial B_{4r}$:
By the maximum principle, it follows that
\[
\limsup_{\Omega\cap B_{4r} \ni y \to x}w(y) \,\leq\, \sup_{\overline \Omega }w \leq \sup_{\partial \Omega }f = v(x)\quad\text{for }\;x \in \Omega \cap \partial B_{4r}.
\]
\end{enumerate}

By the comparison principle, we conclude that $w \leq v$ in $\Omega \cap B_{4r}$.
Taking the supremum over $\mathfrak{L}_f$, we obtain
\[
\underline H_f \le v \quad\text{ in }\;\Omega \cap B_{4r}.
\]
Taking the supremum over $\Omega \cap B_t$ in this inequality and applying Lemma \ref{lem2119thu}, we get 
\begin{equation}	\label{bdry_upper}
\sup_{\Omega \cap B_t} \underline{H}_f \leq \sup_{\partial \Omega \cap B_{4r}} f + \left(\sup_{\partial \Omega}f -\sup_{\partial \Omega \cap B_{4r}}f\right)\exp\left\{-\gamma\sum_{k=0}^{\lfloor \log_2(r/t) \rfloor}\capacity(\overline B_{2^{-k}r}\setminus \Omega, B_{2^{2-k}r})\right\}.
\end{equation}
Using $\underline H_{f}=-\overline H_{-f}$, we also obtain
\begin{equation}	\label{bdry_lower}
\inf_{\Omega \cap B_t} \overline{H}_f \geq \inf_{\partial \Omega \cap B_{4r}} f + \left(\inf_{\partial \Omega}f  - \inf_{\partial \Omega \cap B_{4r}}f \right) \exp\left\{-\gamma\sum_{k=0}^{\lfloor \log_2(r/t) \rfloor}\capacity(\overline B_{2^{-k}r}\setminus \Omega, B_{2^{2-k}r})\right\}.
\end{equation}
Finally, combining \eqref{bdry_upper} and \eqref{bdry_lower} yields the desired bound \eqref{bdry_osc}.
\end{proof}

The following  proposition establishes H\"older continuity estimates for solutions to the Dirichlet problem at a boundary point $x_0$ when the boundary data is H\"older continuous at $x_0$ and a certain capacity condition is satisfied. It also provides global H\"older continuity estimates under a stronger assumption on the capacity.
The proof uses an argument originally due to Maz'ya \cite{Mazya1963}.

\begin{proposition}		\label{prop_holder}
Assume that Conditions \ref{cond1} and \ref{cond2} hold.
Let $f \in C(\partial \Omega)$, let $u$ be the Perron solution to the Dirichlet problem
\[
Lu=0 \;\text{ in }\;\Omega, \quad u=f \;\text{ on }\;\partial\Omega.
\]
If $x_0 \in \partial\Omega$ satisfies the capacity condition
\begin{equation}	\label{eq_rc1}
\liminf_{t \searrow 0}\, \capacity \left(\overline B_t(x_0) \setminus \Omega,\, B_{4t}(x_0)\right)>0,
\end{equation}
and if $f$ is H\"older continuous at $x_0$, then $u$ is also H\"older continuous at $x_0$ (possibly with a different H\"older exponent).
Furthermore, if there exist constants $\varepsilon_0$, $r_0>0$ such that
\begin{equation}	\label{eq_rc2}
\capacity \left(\overline B_t(x_0) \setminus \Omega, B_{4t}(x_0)\right) \ge \varepsilon_0,\quad \forall t \in (0, r_0),
\end{equation}
for all $x_0 \in \partial\Omega$, and if $f \in C^{\beta}(\partial\Omega)$ for some $\beta \in (0,1)$, then 
$u \in C^\alpha(\overline\Omega)$ for some $\alpha \in (0,\beta)$.
\end{proposition} 

\begin{proof}
The assumption \eqref{eq_rc1} implies that there exist constants $\varepsilon_0$, $r_0>0$ such that \eqref{eq_rc2} holds.
By Lemma \ref{lem_bdyest} and \eqref{eq_rc2}, we obtain
\begin{align*}
\osc_{\Omega \cap B_t(x_0)} u &\le \osc_{\partial \Omega \cap B_{4r}(x_0)} f + 2\left(\sup_{\partial \Omega} \,\abs{f}\right)\exp\left\{-\gamma \sum_{k=0}^{\lfloor \log_2(r/t) \rfloor} \capacity\left(\overline B_{2^{-k}r}(x_0) \setminus \Omega, B_{2^{2-k}r}(x_0)\right)\right\}\\
& \le \osc_{\Omega \cap B_{4r}(x_0)} f + C\sup_{\partial \Omega}\, \abs{f}\, \left(\frac{t}{r}\right)^\mu,\qquad 0<t<r<r_0,
\end{align*}
where $\mu=\gamma \varepsilon_0/\log 2>0$.

Suppose $f$ is $\beta$-H\"older continuous at $x_0$.
By taking $r = t^{\mu/(\beta+\mu)}$, we conclude that $u$ is H\"older continuous at $x_0$ with exponent $\alpha_0=\frac{\beta\mu}{\beta+\mu}$.

If \eqref{eq_rc2} holds uniformly for all $x_0 \in \partial\Omega$, then the preceding argument yields
\begin{equation}		\label{eq0806sat}
\osc_{\Omega \cap B_t(x_0)} u \le C \norm{f}_{C^\beta(\partial\Omega)} \, t^{\alpha_0}, \quad \forall x_0\in \partial\Omega.
\end{equation}
Initially, \eqref{eq0806sat} holds for $0<t<r_0$;  however, by allowing $C$ to depend on $r_0$, we extend its validity to all $0<r<\diam \Omega$.
Furthermore, since $u \in W^{2,p_0/2}_{\rm loc}(\Omega)$, it follows that $u$ is locally H\"older continuous in $\Omega$ with exponent $2-4/p_0$.
Define
\[
\alpha=\min(\alpha_0, 2-4/p_0).
\]
For $B_{2r}\subset \Omega$, applying standard $L^p$ estimates,  the Sobolev embedding theorem, interpolation inequalities, and the maximum principle, we obtain
\begin{equation}		\label{eq1440sat}
r^{\alpha} [u]_{C^\alpha(B_r)} \le r^{2-4/p_0} [u]_{C^{2-4/p_0}(B_r)} \le C \norm{u}_{L^\infty(B_{2r})} \le C \norm{f}_{C^\beta(\partial\Omega)}.
\end{equation}

The global H\"older estimates follow by combining \eqref{eq0806sat} and \eqref{eq1440sat} as follows.
Let $x$, $y \in \Omega$ and denote $d(x)= \dist(x,\partial\Omega)$.
Let $x_0 \in \partial \Omega$ be such that $d(x)=\abs{x-x_0}$.
Temporarily assuming $c \equiv 0$, we may, without loss of generality, set $u(x_0)=0$ by replacing $u$ with $u-u(x_0)$ and $f$ with $f-f(x_0)$.
Note that
\[
\norm{f-f(x_0)}_{C^\beta(\partial\Omega)} \le 2 \norm{f}_{C^\beta(\partial\Omega)}.
\]

We consider two cases:

\medskip
\noindent
\textbf{Case 1:} $d(x)>2 \abs{x-y}$.
\smallskip

Applying the interior estimate \eqref{eq1440sat} in $B_r(x)$ with $r=\frac{1}{2}d(x)$ gives
\begin{equation}			\label{eq1011sun}
\frac{\abs{u(x)-u(y)}}{\abs{x-y}^{\alpha}} \le C r^{-\alpha} \norm{u}_{L^\infty(B_{2r}(x))} .
\end{equation}
By \eqref{eq0806sat} and recalling that $\alpha\le \alpha_0$, we obtain
\[
r^{-\alpha} \norm{u}_{L^\infty(B_{2r}(x))} \le r^{-\alpha} \osc_{B_{4r}(x_0)} u \le C  \norm{f}_{C^\beta(\partial\Omega)}.
\]
Combining this with \eqref{eq1011sun}, we obtain
\[
\frac{\abs{u(x)-u(y)}}{\abs{x-y}^{\alpha}} \le C \norm{f}_{C^\beta(\partial\Omega)}.
\]

\medskip
\noindent
\textbf{Case 2:} $d(x)\le 2 \abs{x-y}$.
\smallskip

Setting $r=\abs{x-y}$, note that $x$, $y \in B_{3r}(x_0)$.
By \eqref{eq0806sat}, we obtain
\[
\frac{\abs{u(x)-u(y)}}{\abs{x-y}^{\alpha}} \le r^{-\alpha} \osc_{\Omega \cap B_{3r}(x_0)} u \le  C \norm{f}_{C^\beta(\partial\Omega)}.
\]

\smallskip
Therefore, considering both cases and talking the limits on $\partial\Omega$, we obtain 
\[
[u]_{C^\alpha(\overline \Omega)} \le C \norm{f}_{C^\beta(\partial\Omega)}.
\]

Now, we consider the general case when $c \not\equiv 0$.
Let $v=\zeta u$, where $\zeta$ is as defined in \cite[Proposition 3.3]{DKK25}.
Note that $v$ satisfies (see \cite[(3.10) -- (3.11)]{DKK25})
\[
a^{ij}D_{ij}v+(b^i+2a^{ij}D_j \zeta/\zeta)D_i v=0\;\text{ in }\;\Omega,\quad v=f/\zeta \;\text{ on }\;\partial\Omega.
\]
Using the preceding argument and the properties of $\zeta$, we conclude that
\[
u =v/\zeta \in C^\alpha(\overline\Omega).\qedhere
\] 
\end{proof}

We present examples that satisfy the hypotheses of Theorem \ref{prop_holder}. The following lemma will be useful for this purpose.
\begin{lemma}			\label{lem1329wed}
Let $B_{8r}(x_0) \subset \mathcal{B}$, and let $K$ be a compact subset of $\overline{B_r}(x_0)$. Then,
\[
\frac{1}{C} \capacity^\Delta (K,B_{4r}(x_0)) \le \capacity^\Delta(K,B_{8r}(x_0)) \le C \capacity^\Delta(K,B_{4r}(x_0)),
\]
where $C>0$ is a constant independent of $r$.
\end{lemma}
\begin{proof}
Denote $B_r=B_r(x_0)$.
By Theorems~\ref{thm_green_function} and \ref{thm1929}, it suffices to show the following (cf. Lemma \ref{lem0956fri}):
\begin{equation}	\label{claim250319}
\sup_{\partial B_{2r}} \hat u_{K,B_{4r}}  \le \sup_{\partial B_{2r}} \hat u_{K,B_{8r}} \le 2\sup_{\partial B_{2r}} \hat u_{K,B_{4r}}.
\end{equation}
The first inequality in \eqref{claim250319} follows directly from Definition \ref{def_rfn}.
Define
\[
M := \sup_{\partial B_{4r}}\hat{u}_{K,B_{8r}}.
\]
By the comparison principle, we obtain (by comparing on $\partial B_{4r}$ and $\partial K$)
\begin{equation}	\label{eq1204tue}
\hat u_{K,B_{8r}} \leq (1-M)\hat u_{K,B_{4r}} + M\quad \mbox{in }B_{4r}\setminus K.
\end{equation}
Define
\[
v(x) := \frac{1}{\log 4}\, \log\left(\frac{8r}{\abs{x-x_0}}\right).
\]
Note that $\Delta v=0$ in $B_{8r}\setminus B_{2r}$, with $v =0$ on $\partial B_{8r}$ and $v =1$ on $\partial B_{2r}$.
Applying the comparison principle, we obtain
\[
\hat u_{K,B_{8r}} \le \left( \sup_{\partial B_{2r}}\hat{u}_{K,B_{8r}}\right) v(x) \quad\mbox{on }\;  B_{8r}\setminus \overline{B_{2r}}.
\]
In particular, by taking supremum over $\partial B_{4r}$ in the above, we obtain
\begin{equation}		\label{eq1205tue}
M \le \frac{1}{2} \sup_{\partial B_{2r}} \hat{u}_{K,B_{8r}}.
\end{equation}
Taking supremum over $\partial B_{2r}$ in \eqref{eq1204tue} and using \eqref{eq1205tue}, we deduce
\[
\sup_{\partial B_{2r}} \hat{u}_{K,B_{8r}} \le \sup_{\partial B_{2r}} \hat u_{K,B_{4r}}+\frac{1}{2}\sup_{\partial B_{2r}}\hat{u}_{K,B_{8r}},
\]
which proves the second inequality of \eqref{claim250319}.
\end{proof}

\begin{definition}		%\label{def_line_seg}
A subset $I \subset \bR^2$ is called a line segment if it can be expressed as
\[
I=\{x: x=x_0+t \vec e, \, 0\le t \le \ell\}
\]
for some point $x_0 \in \bR^2$, a unit vector $\vec e \in \bR^2$ with $\abs{\vec e}=1$, and a positive scalar $\ell>0$.
We say that the exterior line segment condition holds at $x_0 \in \partial\Omega$ if there exists a line segment $I \in \bR^2\setminus \Omega$ that starts at $x_0$, i.e., there exists
\[
I=\{x: x=x_0+t \vec e, \, 0\le t \le \ell\}\quad \text{with}\quad I \subset \bR^2\setminus \Omega.
\]
Furthermore, we say that $\Omega$ satisfies the uniform exterior line segment condition if there exists a uniform constant $r_0>0$ such that the exterior line segment condition holds at every boundary point $x_0\in\partial\Omega$ with a line segment of length $\ell \ge r_0$.
\end{definition}

\begin{proposition}		\label{prop0115}
Suppose the exterior line segment condition holds at $x_0 \in \partial \Omega$, then \eqref{eq_rc1} holds.
If $\Omega$ satisfies the uniform exterior line segment condition, the \eqref{eq_rc2} holds.
\end{proposition}
\begin{proof}
For compact subsets $E_i$ and $K$ of $B$, the following properties hold (see \cite{GW82}):
\begin{enumerate}
\item $\capacity^\Delta(E_1,B) \le \capacity^\Delta(E_2,B)$ if $E_1 \subset E_2$.
\item $\capacity^\Delta(\bigcup_{i=1}^N E_i,B) \le \sum_{i=1}^N \capacity(E_i,B)$.
\item $\capacity^\Delta(K,B) = \capacity^\Delta(\partial K,B)$.
\end{enumerate}

Let $I$ be a line segment of length $r$ contained in $\overline{B_r}(x_0) \subset \bR^2$, and let $I_1$, $I_2$, $I_3$ be line segments congruent to $I$ that form the sides of an equilateral triangle $T$ inside $B_r=B_r(x_0)$.
Let $D$ be the closed disk inscribed in $T$.
By Lemma \ref{capa_B_r}, we obtain
\begin{equation}			\label{eq1745wed}
c \le \capacity^\Delta(D,B_{4r}) \le \capacity^\Delta(T,B_{4r}) = \capacity^\Delta(\partial T,B_{4r}) \le \sum_{i=1}^3\capacity^\Delta(I_i,B_{4r}),
\end{equation}
where $c>0$ is a constant independent of $r$.

Next, for any $y_0 \in B_r(x_0)$, observe that
\[
B_{4r}(x_0) \subset B_{6r}(y_0) \subset B_{8r}(x_0).
\]
Thus, applying Lemma \ref{lem1329wed} along with Lemmas \ref{lem1013sat} and \ref{lem0956fri}, we obtain that for any compact set $K \subset \overline{B_r}(x_0)$,
\[
\capacity^\Delta(K,B_{4r}(x_0))  \lesssim  \capacity^\Delta(K,B_{6r}(y_0)) \lesssim \capacity^\Delta(K,B_{8r}(x_0)) \lesssim \capacity^\Delta(K,B_{4r}(x_0)).
\]
This, together with Lemma \ref{lem1329wed}, implies that
\begin{equation}			\label{eq1031mon}
\capacity^\Delta(K,B_{4r}(x_0)) \simeq \capacity^\Delta(K,B_{6r}(y_0)) \simeq \capacity^\Delta(K,B_{4r}(y_0)).
\end{equation}
By the translation and rotation invariance of the Laplace operator, we conclude that
\[
\capacity^{\Delta}(I, B_{4r}) \simeq \capacity^{\Delta}(I_i, B_{4r}),\quad i=1,2,3.
\]
Applying \eqref{eq1745wed}, we conclude that $\capacity(I, B_{4r}) \ge C$ for some constant $C>0$ independent of $r$.
This completes the proof.
\end{proof}

Modifying a result from classical potential theory (see \cite[Theorem 5.5.9]{AG2001}), we can significantly generalize the (uniform) exterior line segment condition in Proposition \ref{prop0115} and obtain the following theorem.

\begin{theorem}		\label{thm0738}
Assume that Conditions \ref{cond1} and \ref{cond2} hold.
Let $\Omega$ be a bounded regular domain in $\mathbb{R}^2$, and let $f \in C(\partial \Omega)$.
Consider the Dirichlet problem
\[
Lu=0 \;\text{ in }\;\Omega, \quad u=f \;\text{ on }\;\partial\Omega,
\]
where $u \in W^{2, p_0/2}_{\rm loc}(\Omega) \cap C(\overline\Omega)$ is the solution.
\begin{enumerate}[leftmargin=*]
\item[(i)]
Suppose $x_0 \in \partial \Omega$ is in a connected component of $\mathbb{R}^2 \setminus \Omega$ that contains at least one other point distinct from $x_0$.
If $f$ is H\"older continuous at $x_0$, then $u$ is also H\"older continuous at $x_0$, possibly with a different H\"older exponent.
\item[(ii)]
Suppose there exists a constant $r_0>0$ such that every point $x_0 \in \partial \Omega$ belongs to a connected component of $\mathbb{R}^2 \setminus \Omega$ that contains a point at least  $r_0$ away from $x_0$.
If $f \in C^{\beta}(\partial\Omega)$ for some $\beta \in (0,1)$, then $u \in C^\alpha(\overline\Omega)$ for some $\alpha \in (0,\beta)$.
\end{enumerate}
\end{theorem}
\begin{proof}
Let $x_1 \neq x_0$ be a point in $\bR^2 \setminus \Omega$.
Let $I$ be the line segment connecting $x_0$ and $x_1$, and let $E \subset \bR^2\setminus \Omega$ be a compact, connected set containing both $x_0$ and $x_1$.
We will show that
\[
\capacity^\Delta(I, B) \le \capacity^\Delta(E, B),
\]
where $B=B_{4r}(x_0)$ with $r\ge \abs{x_0-x_1}$.
Then, the theorem follows from Propositions \ref{prop_holder} and \ref{prop0115}.

Without loss of generality, assume $x_0=0$ and that $x_1$ lies on the nonnegative $x$-axis in the $xy$-coordinate system, with $0< \abs{x_1} \le \frac{1}{4}$ and $B=B_1(0)$.
Let $f$ be the circular projection map of $\bR^2$ onto the nonnegative $x$-axis, defined by
\[
f(r\cos \theta, r\sin \theta)=(r,0), \quad\text{where }\; r\in [0,\infty),\; \theta \in [0,2\pi).
\]
Clearly, $f$ is a contraction, i.e.,
\[
\abs{f(x)-f(y)} \le \abs{x-y}.
\]
We closely follow the proof of \cite[Theorem 5.5.9]{AG2001}.
Let $G(x,y)$ denote the Green's function in $B=B_1(0)$, given by
\[
G(x,y)=-\frac{1}{2\pi} \log \abs{y-x}+ \frac{1}{2\pi} \log ( \abs{x}\,\abs{y-x^*}),\quad\text{where }\; x^*=\frac{x}{\abs{x}^2}.
\]
To proceed, it suffices to show that
\begin{equation}			\label{eq1632mon}
G(f(x), f(y)) \ge G(x,y),\quad x \neq y \in B.
\end{equation}
If this holds, the proof follows the same structure as \cite[Theorem 5.5.9]{AG2001}, with the same arguments applying but with $G(x,y)$ in place of $\abs{x-y}^{2-d}$ throughout.

By the rotational invariance of the Green's function, we have $G(0,y)=G(0,f(y))$,  so \eqref{eq1632mon} is clear when  $x=0$ (or $y=0$ by symmetry).
Moreover, by rotational invariance, we may assume in \eqref{eq1632mon} that $x$ lies on the positive $x$-axis, so that $f(x)=x$.
Since 
\[
G(x,y)= -\frac{1}{2\pi} \log\left(\frac{\abs{y-x}}{\abs{y-x^*}} \right)+\frac{1}{2\pi} \log \abs{x},
\]
it suffices to show that if $x \in B_1(0)$ lies on the positive $x$-axis and $\hat y$ is the circular projection of $y \in B_1(0) \setminus \{0\}$ onto the positive $x$-axis, then
\begin{equation}			\label{eq1838mon}
\frac{\abs{y-x}}{\abs{y-x^*}} \ge  \frac{\abs{\hat y-x}}{\abs{\hat y-x^*}}.
\end{equation}
Since $x$ and $x^*=x/\abs{x}^2$ lie on the positive $x$-axis, \eqref{eq1838mon} follows from elementary plane geometry.
This completes the proof.
\end{proof}

\begin{remark}			\label{rmk2001mon}
Theorem \ref{thm0738} implies that if $x_0 \in \partial \Omega$ lies in a connected component of $\mathbb{R}^2 \setminus \Omega$ that contains at least one other point distinct from $x_0$, then the Green's function $G_\Omega(\,\cdot,y)$ is H\"older continuous at $x_0$.
In particular, if there exists a constant $r_0>0$ such that every point $x_0 \in \partial \Omega$ belongs to a connected component of $\mathbb{R}^2 \setminus \Omega$ that contains a point at least  $r_0$ away from $x_0$, then $G(\,\cdot, y)$ is H\"older continuous in $\overline \Omega \setminus B_r(y)$ for any $r>0$.
Moreover, in this setting, we have the estimate
\[
G(x,y) \le C \left(1\wedge \frac{d(x)}{\abs{x-y}}\right)^\alpha \left\{1+ \log \left(\frac{\diam \Omega}{\abs{x-y}}\right)\right\},
\]
for some $\alpha>0$. See the proof of \cite[Theorem 5.2]{CDKK24} for reference.
Furthermore, \cite[Theorem 1.1]{DK21} and the proof of Theorem \ref{thm1127sat} imply that $G(\,\cdot,y)$ is locally BMO.
\end{remark}

%------------------------------------------------------------------------------%

\end{document}